\input ./style/arxiv-vmsta.cfg
\documentclass[numbers,compress,v1.0.1]{vmsta}
\usepackage{vtexbibtags}
\usepackage{amsmath} 
\usepackage{amssymb} 
\numberwithin{equation}{section} 

\volume{5}
\issue{3}
\pubyear{2018}
\firstpage{337}
\lastpage{351}
\aid{VMSTA110}
\doi{10.15559/18-VMSTA110}
\articletype{research-article}

\startlocaldefs
\newcommand{\lleft}{\left}
\newcommand{\rright}{\right}

\allowdisplaybreaks

\newtheorem{theorem}{Theorem}[section]
\newtheorem{corol}{Corollary}[section]

\theoremstyle{definition}
\newtheorem{example}{Example}[section]
\newtheorem{remark}{Remark}[section]
\newtheorem{defin}{Definition}[section]

\endlocaldefs

\begin{document}

\begin{frontmatter}
\pretitle{Research Article}

\title{Multi-condition of stability for nonlinear stochastic
non-autonomous delay differential equation}

\author{\inits{L.}\fnms{Leonid}~\snm{Shaikhet}\ead
[label=e1]{leonid.shaikhet@usa.net}}
\address{School of Electrical Engineering, \institution{Tel-Aviv
University}, Tel-Aviv 69978, \cny{Israel}}



\markboth{L. Shaikhet}{Multi-condition of stability for nonlinear
stochastic non-autonomous delay differential equation}

\begin{abstract}
A nonlinear stochastic differential equation with the order of nonlinearity
higher than one, with several discrete and distributed delays and time
varying coefficients is considered. It is shown that the sufficient
conditions for exponential mean square stability of the linear part of the
considered nonlinear equation 
also are sufficient conditions
for stability in probability of the initial nonlinear equation. Some new
sufficient condition of stability in probability for the zero solution of
the considered nonlinear non-autonomous stochastic differential
equation is
obtained which can be considered as a multi-condition of stability because
it allows to get for one considered equation at once several different
complementary of each other sufficient stability conditions. The obtained
results are illustrated with numerical simulations and figures.
\end{abstract}

\received{\sday{11} \smonth{6} \syear{2018}}
\revised{\sday{27} \smonth{7} \syear{2018}}
\accepted{\sday{27} \smonth{7} \syear{2018}}
\publishedonline{\sday{20} \smonth{8} \syear{2018}}
\end{frontmatter}

\section{Introduction}

Stability problems for non-autonomous systems are very popular in
theoretical researches and applications and are 
difficult enough even in the
deterministic case (see, for instance, \cite
{Ba,BeGy,BeBr,Go,IdKi,Ku,La,PhHa,PhHi,PhNi,PhVi,SoTa}). In this paper
via the general method of the Lyapunov functionals construction \cite
{Sh10,Sh11,Sh13} some new multi-condition of stability in probability
is obtained for the zero solution of a nonlinear stochastic
differential equation with the order of nonlinearity higher than one,
with several discrete and distributed delays and time varying
coefficients. It is shown that the obtained multi-condition of
stability gives for one considered equation at once a set of different
complementary of each other sufficient stability conditions. Note that
other approaches to analyzing stability in random systems are presented
for example in \cite{Cort,Soon}.

Consider the scalar nonlinear stochastic differential equation with
discrete and distributed delays and time varying coefficients
\begin{align}\label{1.1}
&dx(t)+ \Biggl(\sum^n_{k=0}a_k(t)x(t-h_k)+\sum^n_{k=1}\int^t_{t-h_k}b_k(s)x(s)ds+g(t,x_t)\Biggr)dt\nonumber\\
&\quad +\sigma(t)x(t-\tau)dw(t)=0,\nonumber\\
&x(s)=\phi(s)\in H_2, \qquad s\in[-h,0],\qquad h=\max[h_1,\ldots,h_n,\tau ],\\
&|g(t,\varphi)|\le\int_{-h}^0|
\varphi(s)|^\alpha dG(s),\quad \alpha>1, \quad G=\int_{-h}^0dG(s)<\infty.
\nonumber
\end{align}

Here $a_k(t)$, $b_k(t)$, $\sigma(t)$ are bounded functions, $w(t)$ is
the standard Wiener process on a probability space $\{\varOmega, \mathfrak
F, \mathbf P\}$ \cite{GiSk,Sh13}, $H_2$ is a space of $\mathfrak
F_0$-adapted stochastic processes $\phi(s)$, $s\in[-h,0]$,
\begin{equation*}
\begin{gathered}
\|\phi\|_0=\sup_{s\in[-h,0]}| \phi(s)|, \qquad \|\phi\|^2=\sup_{s\in[-h,0]}\mathbf E|\phi(s)|^2,
\end{gathered} %
\end{equation*}
$\mathbf E$ is the mathematical expectation, $h_0=0$, $h_k>0$,
$k=1,\ldots,n$, $\tau\ge0$, $G(t)$ is a nondecreasing function of bounded
variation, the integral with respect to $dG(s)$ is understood in the
Stiltjes sense.

\begin{defin} The zero solution of \xch{Equation}{the equation} \eqref{1.1} is called:
\begin{itemize}
\item[--] mean square stable if for each $\varepsilon>0$ there exists a
$\delta>0$ such that $\mathbf E|x(t,\phi)|^2<\varepsilon$, $t\ge0$,
provided that $\|\phi\|^2<\delta$;
\item[--] asymptotically mean square stable if it is mean square stable and
\begin{equation*}
\lim\limits
_{t\to\infty}\mathbf E|x(t,\phi)|^2=0
\end{equation*}
for each initial function $\phi$;
\item[--] exponentially mean square stable if it is mean square stable
and there exists $\lambda>0$ such that for each initial function $\phi$
there exists $C>0$ (which may depend on $\phi$) such that $\mathbf
E|x(t,\phi)|^2\le Ce^{-\lambda t}$ for $t>0$;
\item[--] stable in probability if for any $\varepsilon_1>0$ and
$\varepsilon_2>0$ there exists $\delta>0$ such\break that the solution
$x(t,\phi)$ of \xch{Equation}{the equation} \eqref{1.1} satisfies the condition\break
$\mathbf P\{\sup_{t\ge0}|x(t,\phi)|>\varepsilon_1\}<\varepsilon_2$ for
any initial function $\phi$ such that $\mathbf P\{\|\phi\|_0<\delta\}=1$.
\end{itemize}
\end{defin}

Consider the stochastic differential equation \cite{GiSk}
\begin{equation}
\label{1.2} dx(t)=a_1(t,x_t)dt+a_2(t,x_t)dw(t),
\end{equation}
where $x(t)\in\mathbf R^n$, $x_t=x(t+s)$, $s\le0$, $a_1(t,\varphi)\in
\mathbf R^n$, $a_2(t,\varphi)\in\mathbf R^{n\times m}$, $w(t)\in\mathbf
R^m$, 
along with some functional $V(t,\varphi):[0,\infty)\times H_2\to\mathbf
R_+$ that can be presented in the form $V(t,\varphi)=V(t,\varphi
(0),\varphi(s))$, $s<0$, and for $\varphi=x_t$ put
\begin{equation}
\label{1.3} %
\begin{gathered} V_\varphi(t,x)=V(t,
\varphi)=V(t,x_t)=V\bigl(t,x,x(t+s)\bigr),
\\
x=\varphi(0)=x(t), \quad s<0. \end{gathered} %
\end{equation}

Denote by $D$ the set of 
functionals, for which the function
$V_\varphi(t,x)$ defined in \eqref{1.3} has a continuous derivative
with respect to $t$ and 
second continuous derivative with respect to $x$.
For functionals from $D$ the generator $L$ of \xch{Equation}{the equation} \eqref{1.2}
has the form \cite{GiSk,Sh13}
\begin{equation}
\label{1.4} %
\begin{aligned}
LV(t,x_t)&=\frac{\partial V_\varphi(t,x(t))}{\partial t}+
\nabla V_\varphi'\bigl(t,x(t)\bigr)a_1(t,x_t)
\\
&\quad +\frac{1}{2}Tr\bigl[a'_2(t,x_t)
\nabla^2V_\varphi\bigl(t,x(t)\bigr)a_2(t,x_t)
\bigr]. \end{aligned} %
\end{equation}

If in \xch{Equation}{the equation} \eqref{1.2} $a_1(t,0)\equiv0$, $a_2(t,0)\equiv0$
then \xch{Equation}{the equation} \eqref{1.2} has the zero solution and the following
theorems hold.
\begin{theorem} \label{T1.1} Let there exist a functional $V(t,\varphi
)\in D$, positive constants $c_1$, $c_2$ and the function $\mu(t)$ such
that the following conditions hold: $\mu(t)\ge c_1$ for $t\ge0$, $\lim_{t\to\infty}\mu(t)=\infty$ and
\begin{equation}
\label{1.5} \mathbf EV(t,x_t)\ge\mu(t)\mathbf E|x(t)|^2,\qquad
\mathbf EV(0,\phi )\le c_2\|\phi\|^2, \qquad
\mathbf ELV(t,x_t)\le0.
\end{equation}
Then the zero solution of \xch{Equation}{the equation} \eqref{1.2} is asymptotically
mean square stable. If, in particular, $\mu(t)=c_1e^{\lambda t}$,
$\lambda>0$, then the zero solution of \xch{Equation}{the equation} \eqref{1.2} is
exponentially mean square stable.
\end{theorem}
\begin{proof} Integrating the last inequality in \eqref{1.5}, we obtain
$\mathbf EV(t,x_t)\le\mathbf EV(0,\phi)$. So,
\begin{equation*}
c_1\mathbf E|x(t)|^2\le\mu(t)\mathbf
E|x(t)|^2\le\mathbf EV(t,x_t)\le \mathbf EV(0,\phi)\le
c_2\|\phi\|^2.
\end{equation*}
It means that the zero solution of \eqref{1.2} is mean square stable.
Besides, from the inequality $\mathbf E|x(t)|^2\le\mu^{-1}(t)\mathbf
EV(0,\phi)$ it follows that the zero solution of \eqref{1.2} is
asymptotically mean square stable or exponentially mean square stable
if $\mu(t)=c_1e^{\lambda t}$. The proof is completed.
\end{proof}
\begin{theorem} \label{T1.2} \cite{Sh13} Let there exist a functional
$V(t,\varphi)\in D$ such that for any solution $x(t)$ of \xch{Equation}{the equation}
\eqref{1.2} the following inequalities hold:
\begin{equation}
\label{1.6} %
\begin{gathered} V(t,x_t)\ge
\;c_1|x(t)|^2,\qquad V(0,\phi)\le\;c_2\|\phi
\|_0^2,
\\
LV(t,x_t)\le\;0,\quad c_1,c_2>0,
\end{gathered} %
\end{equation}
for any initial function $\phi$ such that $\mathbf P(\|\phi\|_0\le\delta
)=1$, where $\delta>0$ is small enough. Then the zero solution of
\xch{Equation}{the equation} \eqref{1.2} is stable in probability.
\end{theorem}
Via Theorems \ref{T1.1}, \ref{T1.2} a construction of stability
conditions for a given stochastic differential equation is reduced to
construction of appropriate Lyapunov functionals. Via the general
method of the Lyapunov functionals construction \cite{Sh10,Sh11,Sh13},
below some multi-condition of stability in probability for the zero
solution of \xch{Equation}{the equation} \eqref{1.1} is obtained.

\section{Exponential mean square stability of the linear equation}

In this section sufficient conditions of exponential mean square
stability are obtained for the linear part of \xch{Equation}{the equation} \eqref{1.1},
i.e., for \xch{Equation}{the equation} \eqref{1.1} with $g(t,x_t)\equiv0$.

Let $n_i$, $i=1,2$, be integers such that $0\le n_i\le n$. Put
\begin{equation}
\label{2.1} %
\begin{aligned}
S(t)&=\sum^{n_1}_{k=0}a_k(t\,{+}\,h_k)\,{+}\sum^{n_2}_{k=1}b_k(t)h_k, \quad\! m_1\,{=}\,\min\{n_1,n_2\}, \quad\! m_2\,{=}\,\max\{n_1,n_2\},
\\
R_k(t,s)&=\lleft\{ %
\begin{array}{@{}ll}
a_k(s+h_k)+(s-t+h_k)b_k(s),& k=1,\ldots,m_1\\
a_k(s+h_k)\quad\text{if}\quad n_1>n_2, & k=m_1+1,\ldots,m_2,\\
(s-t+h_k)b_k(s)\quad\text{if}\quad n_1<n_2, & k=m_1+1,\ldots,m_2,
\end{array} %
\rright.
\\
R(t)&=\sum^{m_2}_{k=1}\int
^t_{t-h_k}|R_k(t,s)|ds,\qquad
I_k(i,j)=\lleft\{ %
\begin{array}{@{}lll}
1&\text{if}&k\in[i,j],\\
0&\text{if}&k\notin[i,j],
\end{array} %
\rright.
\\
R^\lambda_k(t,s)&=\bigl(S(t)-\lambda\bigr)R_k(t,s)I_k(1,m_2)-b_k(s)I_k(n_2+1,n),
\\
P_\lambda(t)&=\lambda R(t)+\sum^n_{i=n_1+1}|a_i(t)|+
\sum^n_{i=n_2+1}\int^t_{t-h_i}|b_i(
\theta)|d\theta,\\
Q_k^\lambda(t,s)&=|R^\lambda_k(t,s)|+P_\lambda(t)|R_k(t,s)|I_k(1,m_2)+R(t)|b_k(s)|I_k(n_2+1,n),\\
F(t,\lambda)&=\lambda-2S(t)+\sum^n_{k=1} \Biggl(\int^t_{t-h_k}|R^\lambda_k(t,s)|ds+\int^{t+h_k}_te^{\lambda h_k}Q_k^\lambda(\theta,t)d\theta \Biggr)\\
&\quad +\sum^n_{k=n_1+1} \bigl(|a_k(t)|+e^{\lambda h_k}\bigl(1+R(t+h_k)\bigr)|a_k(t+h_k)| \bigr) +e^{\lambda\tau}\sigma^2(t+\tau).
\end{aligned} %
\end{equation}

By virtue of $S(t)$ and $R_k(t,s)$ defined in \eqref{2.1} \xch{Equation}{the equation}
\eqref{1.1} can be presented in the form of a neutral type stochastic
differential equation \cite{Sh13}
\begin{equation}
\label{2.2} %
\begin{aligned}
dz(t,x_t)&=\Biggl(-S(t)x(t)-\sum\limits^n_{k=n_1+1}a_k(t)x(t-h_k)\\
&\quad -\sum\limits^n_{k=n_2+1}\int\limits^t_{t-h_k}b_k(s)x(s)ds-g(t,x_t)\Biggr)dt\\
&\quad -\sigma(t)x(t-\tau)dw(t),
\end{aligned} %
\end{equation}
where
\begin{equation}
\label{2.3} z(t,x_t)=x(t)-\sum^{m_2}_{k=1}
\int^t_{t-h_k}R_k(t,s)x(s)ds.
\end{equation}
\begin{theorem}\label{T2.1} If $g(t,x_t)=0$,
\begin{equation}
\label{2.4} \inf_{t\ge0}S(t)>0, \qquad \sup_{t\ge0}R(t)<1,
\end{equation}
and there exists $\lambda>0$ such that $F(t,\lambda)\le0$ then the zero
solution of \xch{Equation}{the equation} \eqref{1.1} is exponentially mean square stable.
\end{theorem}
\begin{proof} Via \eqref{2.4}, the zero solution of the auxiliary
equation $\dot y(t)=-S(t)y(t)$ is exponentially stable and the function
$v(t)=e^{\lambda t}y^2(t)$, $\lambda>0$, is a Lyapunov function for
this equation. Following the procedure of the Lyapunov functionals
construction \cite{Sh10,Sh11,Sh13}, we will construct Lyapunov
functional $V$ for \xch{Equations}{the equation} \eqref{2.2}, \eqref{2.3} in the form
$V=V_1+V_2$, where $V_1(t,x_t)=e^{\lambda t}z^2(t,x_t)$ and the
additional functional $V_2$ will be chosen below. Using \eqref{1.4} and
\eqref{2.2} with $g(t,x_t)=0$, we have
\begin{equation*}
\aligned LV_1(t,x_t)&=e^{\lambda t} \Biggl[\lambda
z^2(t,x_t)+\sigma^2(t)x^2(t-\tau
)-2z(t,x_t) \Biggl(S(t)x(t)&
\\
&\quad +\sum^n_{k=n_1+1}a_k(t)x(t-h_k)+
\sum^n_{k=n_2+1}\int^t_{t-h_k}b_k(s)x(s)ds
\Biggr) \Biggr]. \endaligned
\end{equation*}
Calculating and estimating $z^2(t,x_t)$ via \eqref{2.3}, \eqref{2.1},
one can show that
\begin{equation*}
\aligned LV_1(t,x_t)&\le e^{\lambda t} \Biggl[ \Biggl(
\lambda-2S(t)+\sum^n_{k=1}\int
^t_{t-h_k}|R^\lambda_k(t,s)|ds+
\sum^n_{k=n_1+1}|a_k(t)|
\Biggr)x^2(t)
\\
&\quad +\sum^n_{k=1}\int^t_{t-h_k}Q^\lambda_k(t,s)x^2(s)ds+
\sum^n_{k=n_1+1}\bigl(1+R(t)
\bigr)|a_k(t)|x^2(t-h_k)
\\
&\quad +\sigma^2(t)x^2(t-\tau) \Biggr]. \endaligned
\end{equation*}
To neutralize the terms with delays in the estimate of $LV_1$ consider
the additional functional
\begin{equation*}
\aligned V_2(t,x_t)&=\sum^n_{k=1}
\int^t_{t-h_k}\int^{s+h_k}_te^{\lambda
(s+h_k)}Q_k^\lambda(
\theta,s)x^2(s)d\theta ds
\\
&\quad +\int^t_{t-\tau}e^{\lambda(s+\tau)}
\sigma^2(s+\tau)x^2(s)ds
\\
&\quad +\sum^n_{k=n_1+1}\int^t_{t-h_k}e^{\lambda
(s+h_k)}
\bigl(1+R(s+h_k)\bigr)|a_k(s+h_k)|x^2(s)ds.
\endaligned
\end{equation*}

Calculating $LV_2(t,x_t)$, via \eqref{2.1} and $F(t,\lambda)\le0$ for
$V=V_1+V_2$ we obtain $LV(t,x_t)\le e^{\lambda t}F(t,\lambda)x^2(t)\le0$.
So, the constructed functional $V(t,x_t)$ satisfies \xch{Conditions}{the conditions}
\eqref{1.5}. Via Theorem \ref{T1.1} the zero solution of \xch{Equation}{the equation}
\eqref{1.1} with $g(t,x_t)=0$ is exponentially mean square stable. The
proof is completed.
\end{proof}
\begin{corol}\label{C2.1} If \xch{Conditions}{the conditions} \eqref{2.4}, $\sup_{t\ge
0}S(t)<\infty$ and
\begin{equation}
\label{2.5} \aligned &\sup_{t\ge0}\frac{1}{S(t)} \Biggl[
\sum^n_{k=1} \Biggl(\int
^t_{t-h_k}|R^0_k(t,s)|ds +
\int^{t+h_k}_tQ_k^0(
\theta,t)d\theta \Biggr)
\\
&\quad +\sum^n_{k=n_1+1} \bigl(|a_k(t)|+
\bigl(1+R(t+h_k)\bigr)|a_k(t+h_k)| \bigr)+
\sigma ^2(t+\tau) \Biggr]<2 \endaligned
\end{equation}
hold then the zero solution of \xch{Equation}{the equation} \eqref{1.1} is
exponentially mean square stable.
\end{corol}
For the proof it is enough to note that \eqref{2.5} is equivalent to
the condition\break $\sup_{t\ge0}F(t,0)<0$ from which it follows that there
exists small enough $\lambda>0$ such that the condition $F(t,\lambda)\le
0$ holds too.

\section{Stability in probability of the nonlinear equation}

In this section it is shown that the sufficient conditions for
exponential mean square stability of the linear part of \xch{Equation}{the equation}
\eqref{1.1} 
also are sufficient conditions for stability in
probability of the initial nonlinear equation.
\begin{theorem}\label{T3.1} Let \xch{Conditions}{the conditions} \eqref{2.4} hold and
there exist $\lambda>0$ and $\varepsilon>0$ such that
\begin{equation}
\label{3.1} F(t,\lambda)+\varepsilon^{\alpha-1}G \Biggl(1+2e^{\lambda h}+
\sum^{m_2}_{k=1}e^{\lambda h_k}\int
^{t+h_k}_t|R_k(\theta,t)|d\theta \Biggr)
\le0,
\end{equation}
where $F(t,\lambda)$ is defined in \eqref{2.1}. Then the zero solution
of \xch{Equation}{the equation} \eqref{1.1} is stable in probability.
\end{theorem}
\begin{proof} Using the functionals $V_1$, $V_2$, defined in the proof
of Theorem \ref{T2.1}, via \eqref{2.2}, \eqref{2.3} we obtain
\begin{equation}
\label{3.2} %
\begin{aligned} L\bigl(V_1(t,x_t)+V_2(t,x_t)
\bigr)&\le e^{\lambda t} \Biggl[F(t,\lambda)x^2(t)-2x(t)g(t,x_t)
\\
&\quad +2\sum^{m_2}_{k=1}\int^t_{t-h_k}R_k(t,s)x(s)dsg(t,x_t)
\Biggr]. \end{aligned} %
\end{equation}
Note that for $|x(s)|\le\varepsilon$, $s\le t$, via \eqref{1.1} and
\eqref{2.1} we have
\begin{equation}
\label{3.3} \aligned 2|x(t)g(t,x_t)|&\le2\int_{-h}^0|x(t)||x(t+s)|^\alpha
dG(s)
\\
&\le\varepsilon^{\alpha-1}\int_{-h}^0
\bigl(x^2(t)+x^2(t+s) \bigr)dG(s)
\\
&=\varepsilon^{\alpha-1} \Biggl(Gx^2(t)+\int
_{-h}^0x^2(t+s)dG(s) \Biggr)
\endaligned
\end{equation}
and
\begin{equation}\label{3.4}
\aligned
2 \Bigg|&\sum^{m_2}_{k=1}\int^t_{t-h_k}R_k(t,s)x(s)dsg(t,x_t)\Bigg|\\
&\quad\le2\sum^{m_2}_{k=1}\int^t_{t-h_k}\int_{-h}^0|R_k(t,s)||x(s)||x(t+\tau)|^\alpha dsdG(\tau)\\
&\quad\le\varepsilon^{\alpha-1}\sum^{m_2}_{k=1}\int^t_{t-h_k}\int_{-h}^0|R_k(t,s)|\bigl(x^2(s)+x^2(t+\tau)\bigr)dsdG(\tau)\\
&\quad =\varepsilon^{\alpha-1} \Biggl(G\sum^{m_2}_{k=1}\int^t_{t-h_k}|R_k(t,s)|x^2(s)ds+R(t)\int_{-h}^0x^2(t+\tau)dG(\tau) \Biggr).
\endaligned
\end{equation}
Substituting \eqref{3.3}, \eqref{3.4} into \eqref{3.2}, we obtain
\begin{equation*}
\aligned L\bigl(V_1(t,x_t)+V_2(t,x_t)
\bigr)&\le e^{\lambda t} \Biggl[F(t,\lambda )x^2(t)+
\varepsilon^{\alpha-1} \Biggl(Gx^2(t)
\\
&\quad+\bigl(1+R(t)\bigr)\int_{-h}^0x^2(t+
\tau)dG(\tau)
\\
&\quad+G\sum^{m_2}_{k=1}\int
^t_{t-h_k}|R_k(t,s)|x^2(s)ds
\Biggr) \Biggr]. \endaligned
\end{equation*}
Using the additional functional
\begin{equation*}
\aligned V_3(t,x_t)&=\varepsilon^{\alpha-1}
\Biggl(2\int_{-h}^0\int_{t+s}^te^{\lambda(\tau+h)}x^2(
\tau)d\tau dG(s)
\\
&\quad +G\sum^{m_2}_{k=1}\int^t_{t-h_k}
\int^{s+h_k}_te^{\lambda(s+h_k)}|R_k(\theta
,s)|x^2(s)d\theta ds \Biggr) \endaligned
\end{equation*}
with
\begin{equation*}
\aligned LV_3(t,x_t)&=\varepsilon^{\alpha-1}e^{\lambda t}
\Biggl[2\int_{-h}^0\bigl(e^{\lambda h}x^2(t)-e^{\lambda(s+h)}x^2(t+s)
\bigr)dG(s)
\\
&\quad +G\sum^{m_2}_{k=1} \Biggl(e^{\lambda h_k}
\int^{t+h_k}_t|R_k(\theta ,t)|d\theta
x^2(t)
\\
&\quad -\int^t_{t-h_k}e^{\lambda(s+h_k-t)}|R_k(t,s)|x^2(s)ds
\Biggr) \Biggr], \endaligned
\end{equation*}
for the functional $V=V_1+V_2+V_3$ via \eqref{3.1} we obtain
\begin{equation*}
\aligned
&LV(t,x_t)\\
&\quad  \le e^{\lambda t} \Biggl[F(t,\lambda)+\varepsilon^{\alpha-1}G
\Biggl(1+2e^{\lambda h}+\sum^{m_2}_{k=1}e^{\lambda h_k}
\int^{t+h_k}_t|R_k(\theta,t)|d\theta
\Biggr) \Biggr]x^2(t)
\\
&\quad  \le0. \endaligned
\end{equation*}
So, the constructed functional $V(t,x_t)$ satisfies \xch{Conditions}{the conditions} \eqref{1.6}. Via Theorem \ref{T1.2} the zero solution of \xch{Equation}{the equation}
\eqref{1.1} is stable in probability. The proof is completed.
\end{proof}
\begin{corol} \label{C3.1} If \xch{Conditions}{the conditions} \eqref{2.4}, $\sup_{t\ge
0}S(t)<\infty$ and \eqref{2.5} hold then the zero solution of \xch{Equation}{the equation} \eqref{1.1} is stable in probability.
\end{corol}
For the proof it is enough to note that \eqref{2.5} is equivalent to
the condition\break $\sup_{t\ge0}F(t,0)<0$ from which it follows that there
exist small enough $\lambda>0$ and $\varepsilon>0$ such that \xch{Condition}{the condition} \eqref{3.1} holds.
\begin{remark} \label{R3.1} From $0\le n_i\le n$, $i=1,2$, it follows
that the couple ($n_1,n_2$) in \xch{Equation}{the equation} \eqref{2.2} has $(n+1)^2$
different values. Thus, Theorem \ref{3.1} generally speaking gives
$(n+1)^2$ different stability conditions at once. Some of these
conditions can be infeasible, from some of these conditions can follow
some other conditions, the remaining conditions will complement 
each other.
\end{remark}

\section{Particular cases of stability condition \eqref{2.5}}

Following Remark \ref{R3.1} let us consider some of possible values of
the couple $(n_1,n_2)$ and obtain appropriate different stability conditions.

If $n_1=n_2=0$ then via \eqref{2.1} $m_1=m_2=0$, $S(t)=a_0(t)$,
$R_k(t,s)=0$, $R(t)=0$, $Q^0_k(t,s)=|R^0_k(t,s)|=|b_k(s)|$, and \xch{Condition}{the condition} \eqref{2.5} takes the form
\begin{equation}\label{4.1}
\sup_{t\ge0}\frac{1}{a_0(t)} \Biggl[
\sum^n_{k=1} \Biggl(|a_k(t)|+|a_k(t+h_k)|+|b_k(t)|h_k+
\int^t_{t-h_k}|b_k(s)|ds \Biggr)+\sigma^2(t+\tau) \Biggr]<2.
\end{equation}

If $n_1=n$, $n_2=0$ then via \eqref{2.1} $m_1=0$, $m_2=n$, and \xch{Condition}{the condition} \eqref{2.5} gives
\begin{equation}\label{4.2}
\aligned
\sup_{t\ge0}\frac{1}{S_0(t)}& \Biggl[\sum^n_{k=1} \Biggl(\int^t_{t-h_k}|S_0(t)a_k(s+h_k)-b_k(s)|ds\\
&+\int^{t+h_k}_t|S_0(\theta)a_k(t+h_k)-b_k(t)|d\theta\\
&+|b_k(t)|\int^{t+h_k}_tA_0(\theta)d\theta+|a_k(t+h_k)|\int^{t+h_k}_tB_0(\theta)d\theta \Biggr)\\
&+\sigma^2(t+\tau) \Biggr]<2, \endaligned
\end{equation}
where
\begin{equation*}
\begin{gathered} S_0(t)=\sum
^n_{k=0}a_k(t+h_k),\qquad
A_0(t)=\sum^n_{k=1}\int
^{t+h_k}_t|a_k(s)|ds,
\\
B_0(t)=\sum^n_{k=1}\int
^t_{t-h_k}|b_k(s)|ds. \end{gathered}
\end{equation*}

If $n_1=0$, $n_2=n$ then $m_1=0$, $m_2=n$, and from \xch{Condition}{the condition} \eqref{2.5} we obtain
\begin{equation}
\label{4.3} \aligned \sup_{t\ge0} \Biggl[\frac{1}{S_1(t)}&
\Biggl(\sum^n_{k=1}|b_k(t)|
\int^{t+h_k}_t \Biggl(S_1(\theta) +\sum
^n_{i=1}|a_i(\theta)| \Biggr)
(t-\theta +h_k)d\theta
\\
&+\sum^n_{k=1}\bigl(|a_k(t)|+
\bigl(1+B_1(t+h_k)\bigr)|a_k(t+h_k)|
\bigr)+\sigma^2(t+\tau ) \Biggr)
\\
&+B_1(t) \Biggr]<2, \endaligned
\end{equation}
where
\begin{equation*}
S_1(t)=a_0(t)+\sum^n_{k=1}b_k(t)h_k,
\qquad B_1(t)=\sum^n_{k=1}\int
^t_{t-h_k}(s-t+h_k)|b_k(s)|ds.
\end{equation*}
If at last $n_1=n_2=n$ then $m_1=m_2=n$, and \xch{Condition}{the condition} \eqref{2.5}
takes the form
\begin{equation}
\label{4.4} \aligned \sup_{t\ge0}& \Biggl[\sum
^n_{k=1}\int^t_{t-h_k}|a_k(s+h_k)+(s-t+h_k)b_k(s)|ds
\\
&+\frac{1}{S_2(t)} \Biggl(\sum^n_{k=1}
\int^{t+h_k}_tS_2(\theta
)|a_k(t+h_k)+(t-\theta+h_k)b_k(t)|d
\theta
\\
&+\sigma^2(t+\tau) \Biggr) \Biggr]<2, \endaligned
\end{equation}
where $S_2(t)=a_0(t)+\sum^n_{k=1}(a_k(t+h_k)+b_k(t)h_k)$.

Using different other combinations of $n_1$ and $n_2$, one can get
different other stability conditions.

\begin{example} \label{E4.1} To demonstrate a possible connection
between the obtained different stability conditions consider, for the
sake of simplicity, the equation with constant coefficients and without
a non-delay term
\begin{equation}
\label{4.5} %
dx(t)+ \Biggl(ax(t-h)+b\int^t_{t-h}x(s)ds+cx^2(t-h) \Biggr)dt
+\sigma x(t-\tau)dw(t)=0.
\end{equation}
For \xch{Equation}{the equation} \eqref{4.5} $n=1$, so, via Remark \ref{R3.1} there are
4 possible stability conditions.
Since in \xch{Equation}{the equation} \eqref{4.5} $a_0=0$ \xch{Condition}{the condition} \eqref{4.1}
does not hold.

Put $p=\frac{1}{2}\sigma^2$. \xch{Condition}{The condition} \eqref{4.2} gives
$p+|a^2-b|h+a|b|h^2<a$ and can be presented in the form
\begin{equation}\label{4.6}
\aligned
b&>\frac{p-a(1-ah)}{h(1+ah)} \quad\text{if} \quad b\le0,\\
b&>\frac{p-a(1-ah)}{h(1-ah)} \quad\text{if} \quad b\in\bigl(0,a^2\bigr), \quad 0<ah<1,\\
b&<\frac{a(1+ah)-p}{h(1+ah)} \quad\text{if} \quad b\ge a^2.
\endaligned
\end{equation}
\xch{Conditions}{The conditions} \eqref{4.3} take the form
\begin{equation}
\label{4.7} |a|<\frac{bh (1-\frac{1}{2}bh^2 )-p}{1+\frac{1}{2}bh^2},\quad0<bh^2<2.
\end{equation}
Calculating the integrals in \eqref{4.4} separately for $a\ge0$ and
$a<0$, from \xch{Condition}{the condition} \eqref{4.4} we obtain
\begin{equation}
\label{4.8} p<\lleft\{ %
\begin{array}{@{}lll}
(a+bh) (1-ah-\frac{1}{2}bh^2 ) & \xch{\text{if}}{if} & a\ge0,\\
(a+bh) (1-ah-\frac{1}{2}bh^2-\frac{a^2}{b} ) & \xch{\text{if}}{if} & a<0,
\end{array}
\rright.
\quad a+bh>0.
\end{equation}
So, if at least one of \xch{Conditions}{the conditions} \eqref{4.6}--\eqref{4.8} holds
then the zero solution of \xch{Equation}{the equation} \eqref{4.5} is stable in
probability and the zero solution of the linear part ($g(t,x_t)\equiv
0$) of this equation is exponentially mean square stable.

\begin{figure}[h!]
\includegraphics{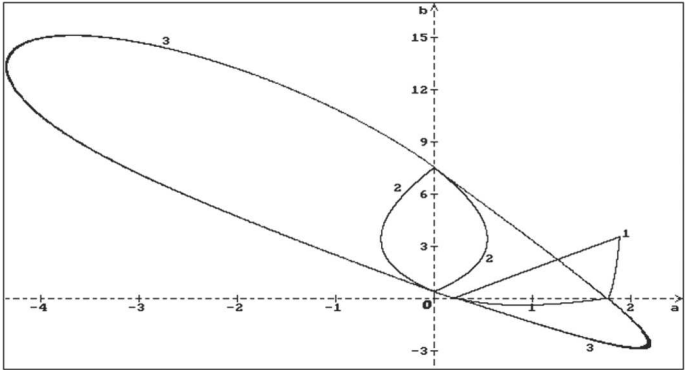}
\caption{Stability regions (1), (2), (3) for \xch{Equation}{the equation} \eqref{4.5},
defined by \xch{Conditions}{the conditions} \eqref{4.6}, \eqref{4.7}, \eqref{4.8}
respectively, for the values of the parameters $h = 0.5$, $p = 0.2$}
\label{Fig1}
\end{figure}
\begin{figure}[h!]
\includegraphics{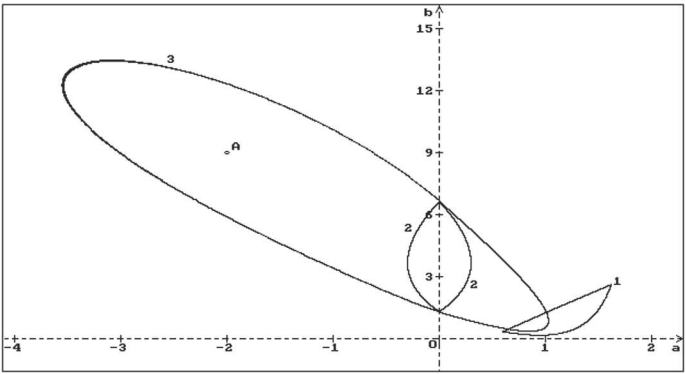}
\caption{Picture similar to Fig.~\ref{Fig1} for the values of the parameters $h =
0.5$, $p = 0.55$}
\label{Fig2}
\end{figure}

\begin{figure}
\includegraphics{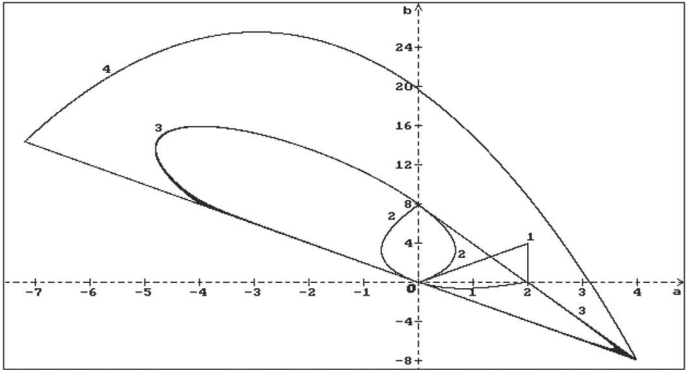}
\caption{Deterministic case ($p = 0$) with $h = 0.5$. The regions (1),
(2), (3) are obtained as in the previous figures, (4) is the exact
stability region given by \eqref{4.10}}
\label{Fig3}
\end{figure}
In Fig.~\ref{Fig1} stability regions for \xch{Equation}{the equation} \eqref{4.5}, given by \xch{Conditions}{the conditions} \eqref{4.6} (the region (1)), \eqref{4.7} (the region (2))
and \eqref{4.8} (the region (3)) are shown in the space of the
parameters $(a,b)$ for $h=0.5$ and $p=0.2$. Note that the regions (1)
and (3) complement of each other but the region (2) is included in the
region (3). It means that \xch{Condition}{the condition} \eqref{4.8} is less
conservative than \eqref{4.7}. Note also that \xch{Condition}{the condition} \eqref{4.8}
coincides with \eqref{4.7} for $a=0$ only. In Fig.~\ref{Fig2} the similar picture
is shown for $h=0.5$ and $p=0.55$.

In the deterministic case ($p=0$) the characteristic equation of the
linear part of \xch{Equation}{the equation} \eqref{4.1} has the form
\begin{equation}
\label{4.9} \omega+ae^{-h\omega}+\frac{b}{\omega}\bigl(1-e^{-h\omega}
\bigr)=0.
\end{equation}
Using $\omega=i\beta$, $i=\sqrt{-1}$, we obtain the system of two
algebraic equations for $a$ and $b$:
\begin{equation*}
a\cos(h\beta)+\frac{b}{\beta}\sin(h\beta)=0,\qquad
a\sin(h\beta)+\frac{b}{\beta}\bigl(1-\cos(h\beta)\bigr)=\beta
\end{equation*}
with the solution
\begin{equation}
\label{4.10} a=\frac{\beta\sin(h\beta)}{1-\cos(h\beta)},\qquad b=-\frac{\beta^2\cos
(h\beta)}{1-\cos(h\beta)}.
\end{equation}
In Fig.~\ref{Fig3} the stability regions (1), (2), (3) obtained respectively from
the sufficient conditions \eqref{4.6}, \eqref{4.7}, \eqref{4.8} in the
deterministic case ($p=0$) are shown for comparison with the exact
stability region (4) given by \xch{Conditions}{the conditions} \eqref{4.10} for $h=0.5$.
The straight line $a+bh=0$ follows from \eqref{4.9} if $\omega\to0$.

In Fig.~\ref{Fig4}, 50 trajectories of the solution of \xch{Equation}{the equation} \eqref{4.5}
are shown at the point $A(-2,9)$ (see Fig.~\ref{Fig2}) for $c=1$, $h=0.5$,
$p=0.55$, $\tau=0$ and the initial function $x(s)=0.6\cos(s)$, $s\in
[-h,0]$. The point $A(-2,9)$ is included in the stability region, thus,
all trajectories converge to zero.
\end{example}
\begin{remark} Suppose that in \xch{Equation}{the equation} \eqref{1.1} discrete delays
are absent, i.e., $a_k(t)=0$, $k=1,\ldots,n$. Then \xch{Conditions}{the conditions} \eqref
{4.2} and \eqref{4.4} coincide respectively with \eqref{4.1} and \eqref
{4.3} and are
\begin{align}
&\sup_{t\ge0}\frac{1}{a_0(t)} \Biggl[\sum^n_{k=1} \Biggl(|b_k(t)|h_k+\int^t_{t-h_k}|b_k(s)|ds \Biggr)+\sigma^2(t+\tau) \Biggr]<2,\label{4.11}\\
&\sup_{t\ge0} \Biggl[\frac{1}{S_1(t)} \Biggl(\sum^n_{k=1}|b_k(t)|\int^{t+h_k}_tS_1(\theta) (t-\theta+h_k)d\theta+\sigma^2(t+\tau) \Biggr)+B_1(t) \Biggr]<2,\label{4.12}\\
&B_1(t)=\sum^n_{k=1}\int^t_{t-h_k}(s-t+h_k)|b_k(s)|ds,\nonumber\\
&S_1(t)=a_0(t)+\sum^n_{k=1}b_k(t)h_k.\nonumber
\end{align}
\end{remark}
\begin{example} \label{E4.2} Consider the stochastic differential
equation \eqref{1.1} with $n=1$, $a_0(t)=a$, $h_1=h$, $b_1(t)=be^{-\mu
t}$, $\mu>0$, $\sigma(t)=\sigma e^{-\nu t}$, $\nu>0$,
$g(t,x_t)=cx^2(t-h)$, i.e.,
\begin{equation}\label{4.13} %
\begin{aligned}
&dx(t)+ \Biggl(ax(t)+b\int
^t_{t-h}e^{-\mu s}x(s)ds+cx^2(t-h)\Biggr)dt \\
&\quad +\sigma e^{-\nu t}x(t-\tau)dw(t)=0.
\end{aligned}
\end{equation}
From \eqref{4.11} we obtain the first condition for stability in
probability of the zero solution of \xch{Equation}{the equation} \eqref{4.13}
\begin{equation}
\label{C1} \frac{1}{2}|b| \biggl(h+\frac{1}{\mu}
\bigl(e^{\mu h}-1\bigr) \biggr)+pe^{-2\nu\tau
}<a, \quad p=
\frac{1}{2}\sigma^2.
\end{equation}
Note that the stability condition \eqref{C1} holds for $a>0$ only.
Using \eqref{4.12} one can get a complementary condition of stability
in probability that holds for $a=0$, $b>0$, $\mu\le2\nu$. Really, in
this case
\begin{equation*}
\begin{gathered} B_1(t)=\frac{b}{\mu^2}e^{-\mu t}
\bigl(e^{\mu h}-1-\mu h\bigr), \qquad S_1(t)=bhe^{-\mu t},
\\
\int^{t+h}_tS_1(\theta) (t-\theta+h)d
\theta=\frac{bh}{\mu^2}e^{-\mu
t}\bigl(e^{-\mu h}-1+\mu h\bigr),
\end{gathered} %
\end{equation*}
and stability condition \eqref{4.12} takes the form
\begin{equation*}
\begin{gathered} \sup_{t\ge0} \biggl[
\frac{b}{\mu^2}e^{-\mu t}\bigl(\cosh(\mu h)-1\bigr)+\frac
{p}{bh}e^{(\mu-2\nu)t}e^{-2\nu\tau}
\biggr]<1,
\\
\cosh(\mu h)=\frac{1}{2}\bigl(e^{\mu h}+e^{-\mu h}\bigr),
\quad p=\frac{1}{2}\sigma^2. \end{gathered} %
\end{equation*}

\begin{figure}[h!]
\includegraphics{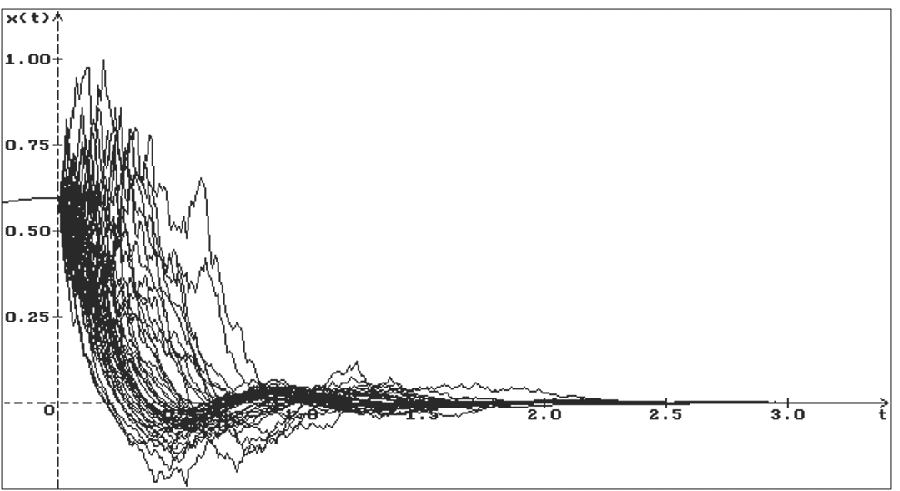}
\caption{50 trajectories of the solution of \xch{Equation}{the equation} \eqref{4.5},
$a=-2$, $b=9$, $c=1$, $h=0.5$, $p=0.55$, $\tau=0$, $x(s)=0.6\cos(s)$, $s\in[-h,0]$}
\label{Fig4}
\end{figure}

\begin{figure}[h!]
\includegraphics{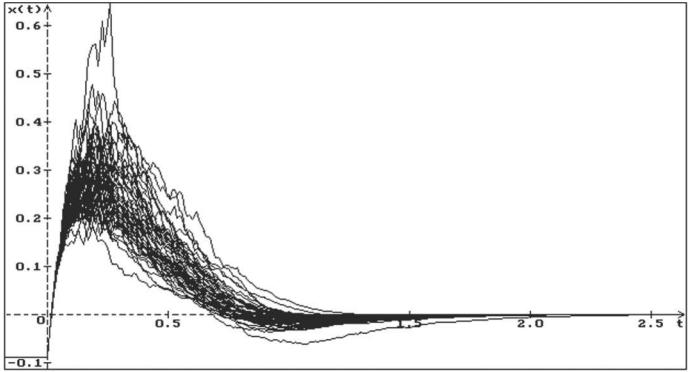}
\caption{50 trajectories of the solution of \xch{Equation}{the equation} \eqref{4.13},
$a=3$, $b=4$, $c=3$, $\mu=0.1$, $\nu=0.01$, $h=0.5$, $p=0.5$, $\tau=0$, $x(s)=-0.09\cos
(s)$, $s\in[-h,0]$}
\label{Fig5}
\end{figure}
\begin{figure}[h!]
\includegraphics{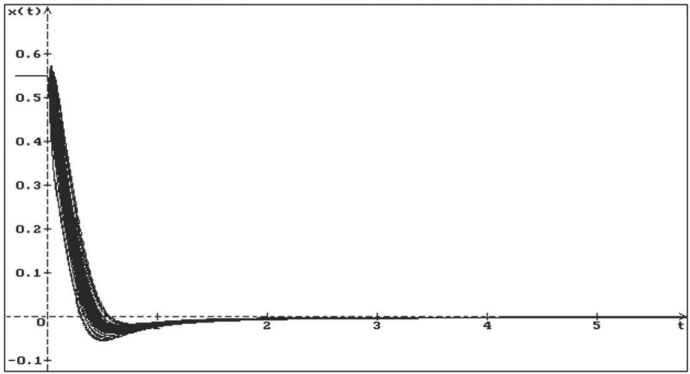}
\caption{50 trajectories of the solution of \xch{Equation}{the equation} \eqref{4.13},
$a=0$, $b=8.5$, $c=1$, $\mu=0.008$, $\nu=0.15$, $h=0.3$, $p=0.2$, $\tau=0$, $x(s)=0.55$, $s\in[-h,0]$}
\label{Fig6}
\end{figure}

Via $\mu\le2\nu$, the supremum is reached at 
$t = 0$, so, we obtain
$\frac{b}{\mu^2}(\cosh(\mu h)-1)+\frac{p}{bh}e^{-2\nu\tau}<1$ or
\begin{equation}
\label{C2} p<bh \biggl(1-b\frac{\cosh(\mu h)-1}{\mu^2} \biggr)e^{2\nu\tau}, \quad
a=0,\quad\mu\le2\nu.
\end{equation}
So, if one of \xch{Conditions}{the conditions} \eqref{C1}, \eqref{C2} holds then the zero
solution of \xch{Equation}{the equation} \eqref{4.13} is stable in probability.

Note that for $\mu=\nu=0$ 
\xch{Conditions}{the conditions} \eqref{C1}, \eqref{C2}
imply,
respectively, two known stability conditions for stochastic
differential equations with constant coefficients $a>|b|h+p$ and
$bh (1-\frac{bh^2}{2} )>p$ (\cite{Sh13}, p.~169).

In Fig.~\ref{Fig5}, 50 trajectories of the solution of \xch{Equation}{the equation} \eqref{4.13}
are shown for $a=3$, $b=4$, $c=3$, $\mu=0.1$, $\nu=0.01$, $h=0.5$,
$p=0.5$, $\tau=0$ and the initial function $x(s)=-0.09\cos(s)$, $s\in
[-h,0]$. The stability condition \eqref{C1} holds, thus all
trajectories converge to zero.
In Fig.~\ref{Fig6}, 50 trajectories of the solution of \xch{Equation}{the equation} \eqref{4.13}
are shown for $a=0$, $b=8.5$, $c=1$, $\mu=0.008$, $\nu=0.15$, $h=0.3$,
$p=0.2$, $\tau=0$ and the initial function $x(s)=0.55$, $s\in[-h,0]$.
The stability condition \eqref{C2} holds and all trajectories converge
to zero.
\end{example}

\section{Conclusions}

In this paper, a nonlinear stochastic non-autonomous differential
equation with discrete and distributed delays and the order of
nonlinearity higher than one is considered. It is shown that
investigation of stability in probability of the nonlinear equation of
such type can be reduced to investigation of exponential mean square
stability of the linear part of the considered equation. A general
multi-condition for stability in probability of the zero solution of
the considered equation is obtained which allows in applications to get
at once a set of different complementary sufficient stability
conditions. Some of these conditions can be infeasible, from some of
these conditions can follow some other conditions, the remaining
conditions will complement 
each other. The idea of construction of
this multi-condition of stability can be used 
also for systems of
nonlinear stochastic differential equations of such type.





\end{document}